\documentclass{amsart}
\usepackage{amsmath,amsthm,amssymb,amscd,setspace}

\usepackage{graphicx}
\usepackage{amsthm}

\newtheorem{thm}{Theorem}[section]
\newtheorem{prop}[thm]{Proposition}
\newtheorem{lem}[thm]{Lemma}
\newtheorem{cor}[thm]{Corollary}

\theoremstyle{definition}

\theoremstyle{remark}
\newtheorem{eg}[thm]{Example}
\theoremstyle{remark}
\newtheorem{rem}[thm]{Remark}

\newcommand{\Z}{\mathbb{Z}}

\newcommand{\Hom}{{\rm Hom}}

\renewcommand{\L}{\mathcal{L}}

\begin{document}

\title{Loop space construction of bigraphs and box complexes}
\author{Takahiro Matsushita}
\email{mtst@math.kyoto-u.ac.jp}

\maketitle

\begin{abstract}
Dochtermann \cite{Dochtermann 2} introduced the loop space construction of a based graph $(G,v)$ whose basepoint is a looped vertex. He showed that the complex $C(\Omega(G,v))$ is homotopy equivalent to the loop space $\Omega(C(G),v)$ of $C(G)$. Here we write $C(G)$ to mean the clique complex of the maximal reflexive subgraph of $G$. In this paper, we consider its bigraph version. A bigraph is a graph equipped with its 2-coloring. We introduce the loop space construction $\Omega_{/K_2}(X,x)$ of a based bigraph $(X,x)$. This is a graph such that $C(\Omega_{/K_2}(X,x))$ is homotopy equivalent to the loop space of the box complex $B_{/K_2}(X)$ of the bigraph. As a result, we have alternative proofs of some results of Matsushita \cite{Matsushita 1} and Schultz \cite{Schultz}.
\end{abstract}

\section{Introduction}

An $n$-coloring of a simple graph $G$ is a map from the vertex set of $G$ to the $n$-point set $\{ 1,\cdots, n\}$ such that adjacent vertices have different values. The chromatic number $\chi(G)$ of $G$ is the smallest number $n$ such that $G$ has an $n$-coloring. The graph coloring problem is to compute the chromatic number of graphs.

Lov\'asz \cite{Lovasz} introduced the neighborhood complex $N(G)$ of a graph $G$, and showed that some homotopy invariant of $N(G)$ is a lower bound for the chromatic number of $G$. The box complex is a $\Z_2$-poset $B(G)$ associated with a graph $G$ (see Section 2), whose classifying space is homotopy equivalent to $N(G)$.

Let $(G,v)$ be a based graph whose basepoint $v$ is a looped vertex. Dochtermann considered a group associated to $(G,v)$, which is similar to the fundamental group of spaces, in a combinatorial way. He introduced the loop space construction of $(G,v)$, and used it to prove the isomorphism between his group and the fundamental group of the clique complex $C(G)$ of the maximal reflexive subgraph of $G$ (see Section 2).

On the other hand, the author \cite{Matsushita 1} considered the 2-fundamental group $\pi_1^2(G,v)$ of a based graph $(G,v)$. This is also a group defined in a combinatorial way, and similar to the fundamental groups of spaces. The 2-fundamental group has a natural subgroup called the even part $\pi_1^2(G,v)_{ev}$, and he showed that the even part and the fundamental group of the neighborhood complex is isomorphic. However, this isomorphism is proved by the comparison with the representations of both of the groups.

It is known that $N(G)$ and $C(G^{K_2})$ are homotopy equivalent (see Section 2). So it is natural to ask that we can show the isomorphism $\pi_1^2(G,v)_{ev} \cong \pi_1(N(G),v)$ in a way similar to Dochtermann \cite{Dochtermann 2}. This is a motivation of this research and in fact we can do it by considering the loop space construction of bigraphs.

A bigraph is a graph $X$ equipped with a 2-coloring $\varepsilon_X$ of $X$ (see \cite{BGZ}). This notion is essential in the research of the box complexes. For a graph $G$, we regard the Kronecker double covering \cite{IP} $K_2 \times G$ as a bigraph by the first projection. The author \cite{Matsushita 3} defined the box complex $B_{/K_2}(X)$ of a bigraph $X$, and showed
$$B_{/K_2}(K_2 \times G) \cong B(G).$$
Moreover, he showed that bigraphs $X$ and $Y$ are isomorphic up to isolated vertices if and only if their box complexes are isomorphic.

A basepoint of a bigraph $X$ is a graph homomorphism $x: K_2 \rightarrow X$ commutative with their 2-colorings. Then $x$ is identified with a point of $B_{/K_2}(X)$. For a based bigraph $(X,x)$, we construct a based graph $\Omega_{/K_2}(X,x)$ and showed the following theorem:

\begin{thm}[Example \ref{eg 4.4}] \label{thm Matsushita}
For a based bigraph $(X,x)$, the clique complex $C(\Omega_{/K_2}(X,x))$ is homotopy equivalent to $\Omega(B_{/K_2}(X), x)$. Moreover, we have
$$\pi_1 (B_{/K_2}(X),x) \cong \pi_1^2(X,x(0))$$
\end{thm}

Using $B_{/K_2}(K_2 \times G) \cong B(G) \simeq N(G)$ and $\pi_1^2(K_2 \times G, (0,v)) \cong \pi_1^2(G,v)_{ev}$, we have a desired isomorphism $\pi_1^2(G,v)_{ev} \cong \pi_1(N(G),v)$.

By similar constructions to $\Omega_{/K_2}(X,x)$, we have an alternative proof of the following theorem by Schultz. Let $X$ be a $\Z_2$-space. We write $\L X$ to mean the free loop space of $X$, and $\L' X$ to mean the space of $\Z_2$-maps from $S^1$ to $X$. Here we consider $S^1$ as a $\Z_2$-space by the antipodal map.

\begin{thm}[Schultz \cite{Schultz}] \label{thm Schultz}
For a graph $G$, there are homotopy equivaleces
$$\lim_{\longrightarrow} \Hom(C_{2r},G) \simeq_{\Z_2} \L B(G), \lim_{\longrightarrow} \Hom(C_{2r+1}, G) \simeq_{\Z_2} \L' B(G).$$
\end{thm}

Here we consider the $\Z_2$-actions on $\L B(G)$ and $\L' B(G)$ as the involutions induced by the reflections of $S^1$.

The rest of this paper is organized as follows. In Section 2, we introduce the notation and the terminology concerning graphs and box complexes. In Section 3, we review the bigraphs and their box complexes, and considered $\times$-homotopy theory \cite{Dochtermann 2} of bigraphs. In Section 4, we recall Quillen's theorem B for posets and prove a slight generalization of it. In Section 5, we introduce the loop (or path) space construction of bigraphs, and show Theorem \ref{thm Matsushita} (Example \ref{eg 4.4}) and Theorem \ref{thm Schultz} (Example \ref{eg 4.5} and Example \ref{eg 4.6}).

\section{Preliminaries}

In this section, we review relevant definitions and introduce the terminology. For an introduction to this subject, we refer to Kozlov \cite{Kozlov book}. For a poset $P$, the classifying space of $P$ (the geometric realization of the order complex) is denoted by $|P|$. We sometimes regard a poset as a topological space by its classifying space. For example, we say that two poset maps are homotopic if the continuous maps induced by them are homotopic.

A {\it graph} is a pair $G= (V(G), E(G))$ consisting of a set $V(G)$ together with a symmetric subset $E(G)$ of $V(G) \times V(G)$. Hence our graphs are undirected, may have loops, but have no multiple edges. A graph $G$ is {\it reflexive} if the diagonal $\Delta_{V(G)}$ is contained in $E(G)$. A {\it graph homomorphism} is a map $f:V(G) \rightarrow V(H)$ such that $(f \times f)(E(G)) \subset E(H)$. For a non-negative integer $n$, the {\it complete graph $K_n$ with $n$-vertices} is the graph defined by $V(K_n) = \{ 0,1,\cdots, n-1\}$ and $E(K_n) = \{ (x,y) \; | \; x \neq y\}$. The category of graphs is denoted by $\mathcal{G}$.

The categorical product $G \times H$ of graphs is defined by
$$V(G \times H) = V(G) \times V(H),$$
$$E(G \times H) = \{ ((x,y),(x',y')) \; | \; (x,x') \in E(G), \; (y,y') \in E(H)\}.$$
We call $K_2 \times G$ the Kronecker double covering \cite{IP} over $G$.

For a vertex $v$ of $G$, let $N(v)$ be the set of vertices adjacent to $v$. For a set $\sigma$ of vertices of $G$, a {\it common neighbor of $\sigma$} is a vertex $v$ with $\sigma \subset N(v)$. The {\it neighborhood complex $N(G)$} is the simplicial complex consisting of finite subsets which have a common neighbor.

The {\it box complex of a graph $G$} is the poset
$$B(G) = \{ (\sigma, \tau) \; | \; \textrm{$\sigma, \tau \in 2^{V(G)} \setminus \{ \emptyset\} $, $\# \sigma , \# \tau < \infty, $ and $\sigma \times \tau \subset E(G).$} \}$$ 
ordered by the product of the inclusion orderings. We regard $B(G)$ as a $\Z_2$-poset whose $\Z_2$-action is the exchange of the first and second entries. In fact there are other definitions of box complexes, and comparisons among them are found in \cite{Zivaljevic}.

\begin{thm}[Babson-Kozlov \cite{BK1}]
There is a natural homotopy equivalence
$$\begin{CD}
|B(G)| @>{\simeq}>> |N(G)|.
\end{CD}$$
\end{thm}

A {\it multi-homomorphism} is a map $\eta : V(T) \rightarrow 2^{V(G)} \setminus \{ \emptyset\}$ such that $\eta (x)$ is finite for every $x \in V(T)$ and $(x,y) \in E(T)$ implies $\eta(x) \times \eta(y) \subset E(G)$. For a pair of multi-homomorphisms $\eta$ and $\eta'$, we write $\eta \leq \eta'$ if $\eta(v) \subset \eta'(v)$ for every $v \in V(T)$. The Hom complex $\Hom(T,G)$ is the poset consisting of the multi-homomorphisms from $T$ to $G$ ordered as above. Clearly, the Hom complex $\Hom(K_2,G)$ is isomorphic to the box complex $B(G)$.

A {\it (reflexive) clique of a graph $G$} is a set $\sigma$ of vertices of $G$ with $\sigma \times \sigma \subset E(G)$. The {\it (reflexive) clique complex of $G$} is a simplicial complex consisting of finite cliques. Note that $\Hom({\bf 1}, G)$ is the face poset of the clique complex of the maximal reflexive subgraph of $G$. So we write $C(G)$ instead of $\Hom({\bf 1}, G)$.

\begin{rem} \label{rem 2.1}
In \cite{Dochtermann 2}, Dochtermann does not assume the finiteness of a value of a multi-homomorphism at each point, and define the ``Hom complex" by the poset of multi-homomorphisms in this sense. However, the homotopy types of these two definitions are naturally homotopy equivalent (see Lemma 4.2 of \cite{Matsushita 2} and its previous paragraph).

We require the finite assumption because we use the following property: For a finite graph $T$, the functor $G \mapsto \Hom(T,G)$ preserves sequentially colimits in our definition.
\end{rem}

We now review some properties of Hom complexes as far as we need.

\begin{prop}[Proposition 3.8 of \cite{Dochtermann 1}] \label{prop 2.1.1}
Let $T$, $G$, and $H$ be graphs. Then there is a natural homotopy equivalence
$$\begin{CD}
\Hom(T,G) \times \Hom(T,H) @>{\simeq}>> \Hom(T,G \times H).
\end{CD}$$
\end{prop}

Let $G$ and $H$ be graphs. Two graph homomorphisms $f$ and $g$ are {\it $\times$-homotopic} (see \cite{Dochtermann 1}) if they belong to the same connected component of $\Hom(G,H)$.

\begin{prop}[Theorem 5.1 of \cite{Dochtermann 1}] \label{prop 2.2}
If graph homomorphisms $f,g : G \rightarrow H$ are $\times$-homotopic, then the induced maps $f_*, g_* : \Hom(T,G) \rightarrow \Hom(T,H)$ are homotopic.
\end{prop}

For a non-negative integer $n$, define the reflexive graph $I_n$ by $V(I_n) = \{ 0,1,\cdots, n\}$ and $E(I_n) = \{ (x,y) \; | \; |x-y| \leq 1\}$. A {\it $\times$-homotopy from $f$ to $g$} is a graph homomorphism $F: G \times I_n \rightarrow H$ $(n \geq 0)$ such that $F(x,0) = f(x)$ and $F(x,n) = g(x)$ for every $x \in V(G)$.

\begin{prop}[Proposition 4.7 of \cite{Dochtermann 1}] \label{prop 2.3}
Graph homomorphisms $f$ and $g$ are $\times$-homotopic if and only if there is a $\times$-homotopy from $f$ to $g$.
\end{prop}

Let $G$ and $H$ be graphs. The {\it exponential graph ${\rm exp}(G,H)$} is defined as follows: A vertex of ${\rm exp}(G,H)$ is a map from $V(G)$ to $V(H)$. Two maps $f$ and $g$ are adjacent if and only if $(f \times g)(E(G)) \subset E(H)$. It is easy to see $\mathcal{G}(T \times G, H) \cong \mathcal{G}(T, H^G)$.

\begin{prop}[Proposition 3.5 of \cite{Dochtermann 1}] \label{prop 2.4}
There is a natural homotopy equivalence
$$\begin{CD}
\Hom(T \times G, H) @>{\simeq}>> \Hom(T, {\rm exp}(G,H)).
\end{CD}$$
\end{prop}

The proof of the following lemma is straightforward, and is omitted.

\begin{lem} \label{lem 2.4}
Let $f$, $g: G \rightarrow H$ be graph homomorphisms, and define the map $F: V(G) \times V(I_1) \rightarrow V(H)$ by $F(x,0) = f(x)$ and $F(x,1) = g(x)$ for $x \in V(G)$. Then $F$ is a graph homomorphism from $G \times I_1$ to $H$ if and only if $(f \times g)(E(G)) \subset E(H)$.
\end{lem}

\section{Bigraphs}

In this section, we introduce the bigraphs and investigate their basic properties. We consider the box complex, Hom complex, and $\times$-homotopy theory of bigraphs.

A {\it bigraph} is a graph $X$ equipped with a 2-coloring $\varepsilon_X : X \rightarrow K_2$. For a bigraphs $X$ and $Y$, a {\it bigraph homomorphism from $X$ to $Y$} is a graph homomorphism $f:X \rightarrow Y$ such that $\varepsilon_Y \circ f = \varepsilon_X$. We write $\mathcal{G}_{/K_2}$ to indicate the category of bigraphs. For a bigraph $X$, set $V_i(X) = \varepsilon^{-1}(i)$ $(i=0,1)$.

Let $X$ and $Y$ be bigraphs. A multi-homomorphism $\eta \in \Hom(X,Y)$ is {\it 2-colored} if $\eta(v) \subset V_i(Y)$ for every $v \in V_i(X)$ $(i =0,1)$. Define the {\it Hom complex $\Hom_{/K_2}(X,Y)$ between bigraphs} to be the induced subposet of $\Hom(X,Y)$ consisting of 2-colored multi-homomorphisms. A bigraph homomorphism is identified with a minimal point of $\Hom_{/K_2}(X,Y)$.

For a bigraph $X$, the {\it box complex $B_{/K_2}(X)$ of $X$ (see \cite{Matsushita 3})} is the poset
$$\{ (\sigma, \tau) \; | \; \textrm{$\sigma \in 2^{V_0(X)} \setminus \{ \emptyset \} $, $\tau \in 2^{V_1(X)} \setminus \{ \emptyset\}, \# \sigma, \# \tau < \infty, $ and $\sigma \times \tau \subset E(X)$.}\}$$
ordered by the product of inclusions. Consider $K_2$ as a bigraph whose 2-coloring is the identity. Then the box complex $B_{/K_2}(X)$ is isomorphic to the Hom complex $\Hom_{/K_2}(K_2, X)$.

Next we consider $\times$-homotopy theory of bigraphs. Let $f,g:X \rightarrow Y$ be bigraph homomorphisms. Then $f$ and $g$ are {\it $\times$-homotopic} if and only if they belong to the same connected component of $\Hom_{/K_2}(X,Y)$, and in this case we write $f \simeq_{\times } g$. A bigraph homomorphism $f: X \rightarrow Y$ is a {\it $\times$-homotopy equivalence} if there is a bigraph homomorphism $g: Y \rightarrow X$ such that $gf \simeq_\times {\rm id}_X$ and $fg \simeq_{\times} {\rm id}_Y$.

\begin{lem}\label{lem 3.2}
Let $f$, $g: Y \rightarrow Z$ be bigraph homomorphisms. If $f \simeq_\times g$, then the induced maps
$$f_*, g_* :\Hom_{/K_2}(X,Y) \rightarrow \Hom_{/K_2}(X,Z)$$
are homotopic for every bigraph graph $X$.
\end{lem}
\begin{proof}
As is the case of the usual Hom complex, we have a composition map
$$\Hom_{/K_2}(Y,Z) \times \Hom_{/K_2}(X,Y) \longrightarrow \Hom_{/K_2}(X,Z), \; (\tau , \eta) \mapsto \tau * \eta$$
defined by
$$(\tau * \eta)(v) = \bigcup_{w \in \eta(v)} \tau(w).$$
Let $\varphi : [0,1] \rightarrow |\Hom_{/K_2}(Y,Z)|$ be a path joining $f$ to $g$. Then the composition
$$\begin{CD}
[0,1] \times |\Hom(X,Y)| @>{\varphi \times {\rm id}}>> |\Hom(Y,Z)| \times |\Hom(X,Y)| @>{|*|}>> |\Hom(X,Z)|
\end{CD}$$
gives a homotopy from $f_*$ to $g_*$.
\end{proof}

A principal example of a $\times$-homotopy equivalence is given by {\it folds} (see \cite{BK1} and \cite{Kozlov}), the deletion of a dismantlable vertex. A vertex $v$ of a bigraph $X$ is {\it dismantlable} if there is $w \in V(X)$ such that $v\neq w$ and $N(v) \subset N(w)$.

\begin{lem}[See Kozlov \cite{Kozlov}] \label{lem 3.2.1}
Let $X$ be a bigraph and $v$ a vertex of $X$. If $v$ is dismantlable, then the inclusion $X \setminus v \hookrightarrow X$ is a $\times$-homotopy equivalence.
\end{lem}
\begin{proof}
Let $i$ be the inclusion $X \setminus v \hookrightarrow X$, and $w$ a vertex of $X$ such that $w \neq v$ and $N(v) \subset N(w)$. Let $r: X \rightarrow X \setminus v$ be a retraction of $i$ which takes $v$ to $w$. Define $\eta \in \Hom(X,X)$ by
$$\eta(x) = \begin{cases}
\{ x \} & (x \neq v)\\
\{ v,w\} & (x = v).
\end{cases}$$
Then we have $ir \leq {\eta}$ and ${\rm id}_X \leq \eta$. Thus we have $ir \simeq_\times {\rm id}_X$ and $ri = {\rm id}_{X \setminus v}$.
\end{proof}

Let $X$ be a bigraph and $G$ a graph (see Section 2). Consider the product $X \times G$ as a bigraph whose 2-coloring is the composition
$$\begin{CD}
X \times G @>{p_1}>> X @>{\varepsilon_X}>> K_2,
\end{CD}$$
where $p_1$ is the first projection.

Let $X$ and $Y$ be bigraphs. Define the graph $Y^X$ to be the induced subgraph of the usual exponential graph (see Section 2) from $X$ to $Y$ whose vertices are maps from $V(X)$ to $V(Y)$ commutative with their 2-colorings.

\begin{lem} \label{lem 3.2.2}
Let $G$ be a graph, and $X$ and $Y$ bigraphs. The following assertions hold:
\begin{itemize}
\item[(1)] There is a natural isomorphism $\mathcal{G}_{/K_2}(X \times G, Y) \cong \mathcal{G}(G, Y^X)$.
\item[(2)] There is a natural homotopy equivalence $\Hom_{/K_2}(X \times G, Y) \simeq \Hom(G, Y^X)$.
\end{itemize}
\end{lem}
\begin{proof}
The proof of (1) is straightforward and is omitted. The proof of (2) is similar to the case of usual Hom complexes (see Proposition 3.5 of Dochtermann \cite{Dochtermann 1}). So we only give a sketch. Define the order-preserving maps $\Phi : \Hom_{/K_2}(G \times X, Y) \longrightarrow \Hom(G,Y^X)$ and $\Psi : \Hom(G,Y^X) \longrightarrow \Hom_{/K_2}(G \times X, Y)$ as follows:
$$\Phi(\eta)(v) = \{ f: V(X) \rightarrow V(Y) \; | \; \textrm{$f(x) \in \eta(v,x) $ for every $x \in V(X)$.}\},$$
$$\Psi (\eta) (v,x) = \{ f(x) \; | \; f \in \eta(v)\}.$$
Then one can show $\Psi \circ \Phi = {\rm id}$ and $\Phi \circ \Psi \geq {\rm id}$.
\end{proof}

\begin{cor} \label{cor 3.2.3}
For a bigraph $X$, we have $C(X^{K_2}) \simeq B_{/K_2}(X)$.
\end{cor}

A $\times$-homotopy from $f$ to $g$ is a bigraph homomorphism $F: X \times I_n \rightarrow Y$ such that $F(x,0) = f(x)$ and $F(x,n) = g(x)$ for all $x \in V(X)$. The following lemma is easily verified and the proof is omitted.

\begin{lem} \label{lem 3.3}
Let $f$, $g:X \rightarrow Y$ be bigraph homomorphisms. Define the map $F: V(X) \times V(I_1) \rightarrow V(Y)$ by $F(x,0) = f(x)$ and $F(x,1) = g(x)$ for $x \in V(X)$. Then $F$ is a bigraph homomorphism from $X \times I_1$ to $Y$ if and only if there is a 2-colored multi-homomorphism $\eta$ with $f \leq \eta$ and $g \leq \eta$.
\end{lem}

\begin{prop} \label{prop TFAE}
Let $f$ and $g$ be bigraph homomorphisms from $X$ to $Y$. Then the following hold:
\begin{itemize}
\item[(1)] $f$ and $g$ are $\times$-homotopic.
\item[(2)] There is a $\times$-homotopy from $f$ to $g$.
\item[(3)] $f$ and $g$ belong to the same connected component of the maximal reflexive subgraph of $Y^X$.
\end{itemize}
\end{prop}
\begin{proof}
Lemma \ref{lem 3.3} implies that (1) and (2) are equivalent. On the other hand, (1) of Lemma \ref{lem 3.2.2} implies that (2) and (3) are equivalent.
\end{proof}

\begin{lem} \label{lem 3.4}
Let $f,g : X \rightarrow Y$ be bigraph homomorphisms which are $\times$-homotopic. For every $2$-colored graph $Z$, the maps $Z^Y \rightarrow Z^X$ are $\times$-homotopic.
\end{lem}
\begin{proof}
We can define the composition map
$$Z^Y \times Y^X \rightarrow Z^X, \; (g,f) \mapsto g \circ f.$$
Because of the equivalence between (1) and (3) of Proposition \ref{prop TFAE}, a similar proof to Lemma \ref{lem 3.2} works.
\end{proof}

We conclude this section with odd involutions of bigraphs \cite{Matsushita 3}.

An odd involution of a bigraph $X$ is a graph homomorphism $\alpha : X \rightarrow X$ such that $\varepsilon_X \circ \alpha(v) \neq \varepsilon_X(v)$ for every $v \in V(X)$. Clearly, an involution is regarded as a $\Z_2$-action, and we write $X/\alpha$ to indicate the quotient graph by the $\Z_2$-action on $X$.

A typical example of the odd involutions is the involution $(0,v) \leftrightarrow (1,v)$ of the Kronecker double covering $K_2 \times G$ over a graph $G$. On the other hand, for an odd involution $\alpha$ of $X$, it is easy to see that $K_2 \times (X/\alpha) \cong X$ as bigraphs.

For later sections, we need the following construction: Let $X$ and $Y$ be bigraphs, and $\alpha_X$ and $\alpha_Y$ odd involutions of $X$ and $Y$, respectively. Then the exponential graph $Y^X$ of bigraphs has the involution
$$\alpha_{Y^X} (f) = \alpha_Y \circ f \circ \alpha_X.$$

\section{Quillen type lemma}

We first recall the Quillen's theorem B for posets:

\begin{thm}[Quillen \cite{Quillen}] \label{thm Quillen B}
Let $p:P \rightarrow Q$ be an order-preserving map. Suppose that for every pair $y,y'$ of elements of $Q$ with $y \leq y'$, the map $p^{-1}(Q_{\leq y}) \hookrightarrow p^{-1}(Q_{\leq y'})$ is a homotopy equivalence. Then the diagram
$$\begin{CD}
|p^{-1}(Q_{\leq y})| @>>> |P|\\
@VVV @VV{p_*}V\\
|Q_{\leq y}| @>>> |Q|
\end{CD}$$
is a homotopy pullback square for every element $y$ of $Q$. In particular, $|p^{-1}(Q_{\leq y})|$ is a homotopy fiber of $p_* : |P| \rightarrow |Q|$ over $y \in Q$.
\end{thm}

We need the following slight generalization of Theorem \ref{thm Quillen B}.

\begin{cor}\label{cor Quillen B}
Let $X$, $Y$, and $Z$ be posets and let $p:Y \rightarrow X$ and $f:Z \rightarrow X$ be order preserving maps. Suppose that the following conditions hold.
\begin{itemize}
\item[(1)] Let $x$ and $x'$ be elements of $X$. If $x \leq x'$, then the inclusion $p^{-1}(X_{\leq x}) \hookrightarrow p^{-1}(X_{\leq x'})$ is a homotopy equivalence.
\item[(2)] For each element $z$ of $Z$, $f$ induces an isomorphism $Z_{\leq z} \rightarrow X_{\leq f(z)}$.
\end{itemize}
Then the diagram
\begin{eqnarray}\label{CD 4.2}
\begin{CD}
|Z \times_{X} Y| @>{|g|}>> |Y|\\
@V{|q|}VV @VV{|p|}V\\
|Z| @>{|f|}>> |X|
\end{CD}
\end{eqnarray}
is homotopy pullback. Here $g: Z \times_X Y \rightarrow Y$ and $q: Z \times_X Y \rightarrow Z$ are the projections.
\end{cor}
\begin{proof}
Set $W = Z \times_{X} Y$. Let $z$ be an element of $Z$. We claim that the diagram (\ref{CD 4.2}) induces a homotopy equivalence from the homotopy fiber of $|q|$ over $z$ to the homotopy fiber of $|p|$ over $f(z)$. Note that
\begin{eqnarray*}
q^{-1}(Z_{\leq z}) & = & \{ (z',y) \; | \; \textrm{$z' \leq z$ and $f(z') = p(y)$.}\}\\
& \cong & \{ y \in Y \; | \; p(y) \leq f(z)\}\\
& = & p^{-1}(X_{\leq f(z)}).
\end{eqnarray*}
by the condition (2). Therefore for a pair $z'$ and $z$ of elements of $Z$, $z' \leq z$ implies that $q^{-1}(Z_{\leq z'}) \hookrightarrow q^{-1}(Z_{\leq z})$ is a homotopy equivalence. Thus $q: W \rightarrow Z$ satisfies the hypothesis of Quillen's theorem B, and the diagram (\ref{CD 4.2}) induces a homotopy equivalence from the homotopy fiber of $|q|$ to the homotopy fiber of $|p|$.

Consider a commutative diagram
\begin{eqnarray}
\begin{CD}
|W| @>j>> W' @>{q'}>> |Z|\\
@V{|g|}VV @VVV @VV{f}V\\
|Y| @>i>> Y' @>{p'}>> |X|
\end{CD}
\end{eqnarray}
such that $q' j = |q|$, $p' i = |p|$, $p'$ and $q'$ are fibrations, and $i$ and $j$ are weak equivalences. Since $p'$ is a fibration, we have that $|Z| \times_{|X|} Y'$ is a homotopy pullback of $|p|$ and $|f|$. Thus it suffices to show that $W' \rightarrow |Z| \times_{|X|} Y'$ is a weak equivalence.

Let $z$ be an element of $Z$. Consider the commutative diagram
$$\begin{CD}
F' @>>> W' @>{q'}>> |Z|\\
@VVV @VVV @|\\
F @>>> |Z| \times_{|X|} Y' @>>> |Z|,
\end{CD}$$
where $F'$ and $F$ are fibers over $z$. Note that $F$ is a homotopy fiber of $|p| : |Y| \rightarrow |X|$, and that $|Z| \times_{|X|} p' : |Z| \times_{|X|} Y' \rightarrow |Z|$ is a fibration since it is a pullback of a fibration. We have already shown that the map $F' \rightarrow F$ is a homotopy equivalence. Thus $W' \rightarrow |Z| \times_{|X|} Y'$ is a weak equivalence.
\end{proof}

\begin{rem} \label{rem 5.3}
In Corollary \ref{cor Quillen B}, suppose that $X$, $Y$, and $Z$ are $\Z_2$-spaces and $p$, $f$ are $\Z_2$-equivariant. If we assume the following additional assumption $(1)'$, then we have that $|W| = |Z \times_X Y|$ is a homotopy pullback of $|f|$ and $|p|$ in the category of $\Z_2$-spaces:
\begin{itemize}
\item[(1)$'$] The map $p^{\Z_2} : Y^{\Z_2} \rightarrow X^{\Z_2}$ satisfies the hypothesis of Quillen's theorem B.
\end{itemize}
Here $X^{\Z_2}$ denotes the induced subposet of $X$ consisting of fixed points. The proof of this fact is obtained by modifying of that of Corollary \ref{cor Quillen B} in a straightforward way, so we omit the details.
\end{rem}

\section{Loop space construction}

In this section, we shall construct the loop space construction of bigraphs. We should note that the following construction is a straightforward generalization of Dochtermann \cite{Dochtermann 2}.

Let $a$ and $b$ be a pair of integers with $a \leq b$. Define the bigraph $L_{a,b}$ by
$$V(L_{a,b}) = \{ x \in \Z \; | \; a \leq x \leq b\},$$
$$E(L_{a,b}) = \{ (x,y) \; | \; |x-y| \leq 1\}$$
with the 2-coloring $L_{a,b} \rightarrow K_2$, $x \mapsto (x \; {\rm mod.} 2)$. Consider the sequence
\begin{eqnarray*}
K_2 = L_{0,1} \hookrightarrow L_{-1,2} \hookrightarrow \cdots \hookrightarrow L_{-n, n+1} \hookrightarrow \cdots .
\end{eqnarray*}
The colimit of this sequence is denoted by $L_{-\infty,\infty}$. Namely, $V(L_{-\infty, \infty}) = \Z$ and $E(L_{-\infty, \infty}) = \{ (x,y) \; | \; |x-y| = 1\}$. Let $r_n : L_{-n-1,n+2} \rightarrow L_{-n,n+1}$ be the retraction. Note that the inclusions in the above sequence are $\times$-homotopy equivalences (Lemma \ref{lem 3.2.1}). Thus the retraction $r_n$ is a $\times$-homotopy equivalence. Consider the sequence
\begin{eqnarray*}
\begin{CD}
X^{K_2} = X^{L_{0,1}} @>{r_0^*}>> X^{L_{-1,2}} @>{r_1^*}>> \cdots @>>> X^{L_{-n,n+1}} @>{r_n^*}>> \cdots
\end{CD}
\end{eqnarray*}
and define the graph $X^L$ to be the colimit of the above sequence.

\begin{lem}
The inclusion $C(X^{K_2}) \hookrightarrow C(X^L)$ is a homotopy equivalence.
\end{lem}
\begin{proof}
Since $r_n^* : X^{L_{-n,n+1}} \rightarrow X^{L_{-n-1, n+2}}$ is a $\times$-homotopy equivalence, the sequence
$$C(X^{K_2}) = C(X^{L_{0,1}}) \hookrightarrow C(X^{L_{-1,2}}) \hookrightarrow \cdots$$
is a sequence of trivial cofibrations. Thus the colimit
$$C(X^{K_2}) \rightarrow {\rm colim}_{n \rightarrow \infty} C(X^{L_{-n, n+1}}) \cong C(X^L)$$
is a homotopy equivalence.
\end{proof}

\begin{rem}
A looped vertex of $X^L$ is a graph homomorphism from $L_{-\infty, \infty}$ to $X$ such that the following properties holds: There is an integer $n$ such that $f(k) = f(k+2)$ if $k \geq -n$ and $f(k) = f(k-2)$ if $k \leq n$. In this sense, we can regard $X^L$ as the graph of ``stable paths" of $X$.
\end{rem}

Define the homomorphisms $\iota_{-2n}, \iota_{2n} : K_2 \rightarrow L_{-2n, 2n+1}$ by $\iota_k(i) = k+i$ for $k= \pm 2n$ and $i = 0,1$. Then we have the homomorphisms
$$\iota_{\pm 2n}^* : X^{L_{-2n,2n+1}} \longrightarrow X^{K_2}, \; f \mapsto f \circ \iota_{\pm 2n}.$$
Let $e_{2n} = \iota_{2n}^*$ and $e_{-2n} = \iota_{-2n}^*$. Then these homomorphisms induce homomorphisms
$$e_{+\infty}, e_{-\infty} : X^L \longrightarrow X^{K_2}.$$
Note that $e_{+\infty}$ and $e_{-\infty}$ are retractions of $X^{K_2} \hookrightarrow X^L$.

The main structural result in this paper is the following:

\begin{prop} \label{prop 4.2}
The order-preserving map
$$(e_{- \infty}, e_{+\infty})_* : C(X^L) \longrightarrow C(X^{K_2} \times X^{K_2})$$
satisfies the hypothesis of Quillen's theorem B (see Theorem \ref{thm Quillen B}).
\end{prop}
\begin{proof}
(The proof given here is essentially the same as Dochtermann \cite{Dochtermann 2}) Recall that we write $V_i(X)$ to mean $\varepsilon^{-1}(i)$ $(i = 0,1)$. Let $E'(X) = E(X) \cap (V_0(X) \times V_1(X))$. Note that a looped vertex of $X^{K_2} \times X^{K_2}$ is identified with a pair of elements of $E'(X)$. Let $\sigma, \sigma' \in C(X^{K_2} \times X^{K_2})$ with $\sigma \leq \sigma'$. Then these are finite sets of looped vertices of $V(X^{K_2}) \times V(X^{K_2})$. Define the induced subgraphs $A_n$ and $A'_n$ of $X^{K_2}$ as follows:
$$V(A_n) = \{ f : L_{-2n, 2n+1} \rightarrow X \; | \; ((f(-2n), f(-2n+1)),(f(2n), f(2n+1))) \in \sigma\},$$
$$V(A'_n) = \{ f : L_{-2n, 2n+1} \rightarrow X \; | \; ((f(-2n), f(-2n+1)),(f(2n), f(2n+1))) \in \sigma'\}.$$
Then $s_n^* = (r_{2n} \circ r_{2n+1} )^* : X^{L_{-2n, 2n+1}} \hookrightarrow X^{L_{-2n-2, 2n+3}}$ induces the inclusions $i_n : A_n \hookrightarrow A_{n+1}$ and $i'_n : A'_n \hookrightarrow A'_{n+1}$. Let $A_{\infty}$ and $A'_{\infty}$ be the colimits of $A_n$ and $A'_n$, respectively. Then we have
$$C(A_\infty) = (e_{-\infty}, e_{+\infty})_*^{-1}(C(X^{K_2} \times X^{K_2})_{\leq \sigma}),\;  C(A'_\infty) = (e_{-\infty}, e_{+\infty})_*^{-1}(C(X^{K_2} \times X^{K_2})_{\leq \sigma'}).$$
Consider the commutative diagram
\begin{eqnarray}
\begin{CD}
A_0 @>{i_0}>> A_1 @>{i_1}>> \cdots @>>> A_n @>{i_n}>> \cdots \\
@V{j_0}VV @V{j_1}VV @. @VV{j_n}V @.\\
A'_0 @>{i'_0}>> A'_1 @>{i'_1}>> \cdots @>>> A'_n @>{i'_n}>> \cdots ,
\end{CD}
\end{eqnarray}
where each arrow in the diagram is an inclusion. Let $((x_0,y_0),(x_1,y_1))$ be an element of $\sigma$. Define the graph homomorphism $s_n : A'_n \rightarrow A_{n+1}$ as follows: Let $\gamma \in V(A'_n) \subset \mathcal{G}_{/K_2}(L_{-2n,2n+1})$, define $u_n(\gamma) : V(L_{-2n -2,2n+3}) \rightarrow V(X)$ by
$$u_n(\gamma) |_{L_{-2n,2n+1}} = \gamma$$
and
$$u_n(-2n-2) = x_0, u_n(-2n-1) = y_0, u_n(\gamma)(2n+2) = x_1, u_n(\gamma)(2n+3) = y_1.$$
Then one can show that $u_n j_n \simeq_{\times} i_n$, $j_{n+1} u_n \simeq_{\times} i'_n$ (Lemma \ref{lem 2.4}). Using this and Proposition \ref{prop 2.2}, one can show that the inclusion $j_\infty : C(A_\infty) \hookrightarrow C(A'_\infty)$ induces isomorphisms between their homotopy groups.
\end{proof}

\begin{rem} \label{rem 4.2.1}
Suppose that $X$ is equipped with an odd involution $\alpha_X$. For each $n$, consider the odd involution $\beta_n$ of $L_{-2n,2n+1}$ defined by $x \mapsto 1-x$. Then $X^{L_{-2n,2n+1}}$ has the natural involution $\alpha_n$, described in the end of Section 3. Then the involutions $\alpha_n$ induce an involution $\alpha_{\infty}$ of $X^L$, and the map $(e_{-\infty}, e_{+\infty}) : X^L \rightarrow X^{K_2} \times X^{K_2}$ is $\Z_2$-equivariant. Here we consider the $\Z_2$-action on $X^{K_2} \times X^{K_2}$ as the exchange of the first and second entries.

We claim that the map $(e_{-\infty}, e_{+\infty})_* : C(X^L) \rightarrow C(X^{K_2} \times X^{K_2})$ satisfies the property (1)$'$ of Remark \ref{rem 5.3}. To see this, we need to show that the restriction of $e_{+\infty}$
$$e_{+\infty *} : C(X^L)^{\Z_2} \longrightarrow C(X^{K_2}) \cong (C(X^{K_2} \times X^{K_2}))^{\Z_2}$$
satisfies the hypothesis of Quillen's theorem B (Theorem \ref{thm Quillen B}).

Let $\sigma, \sigma' \in C(X^{K_2} \times X^{K_2})^{\Z_2}$ with $\sigma \leq \sigma'$. We define $A_n$ and $A'_n$ as the proof of Proposition \ref{prop 4.2}. Set $B_n = A_n^{\Z_2}$ and $B'_n = (A'_n)^{\Z_2}$, i.e. the induced subgraphs of $A_n$ and $A'_n$ consisting of fixed vertices. Let $B_\infty$ and $B'_\infty$ the colimits of $B_n$ and $B'_n$, respectively. Then
$$C(B_{\infty}) = e_{+\infty *}^{-1}(C(X^{K_2})_{\leq \sigma}), \; C(B'_{+\infty}) = e_{\infty *}^{-1}(C(X^{K_2})_{\leq \sigma'}).$$
After that, almost the same proof follows and we omit the details.
\end{rem}

\begin{thm} \label{thm 4.3}
Let $X$ and $Y$ be bigraphs and $f$, $g : Y \rightarrow X$ bigraph homomorphisms. Suppose that either $f$ or $g$ is an inclusion. Define the graph $Z$ by the pullback diagram
$$\begin{CD}
Z @>>> X^L\\
@VVV @VVV\\
Y^{K_2} @>{(f^{K_2}, g^{K_2})}>> X^{K_2} \times X^{K_2}.
\end{CD}$$
Then $C(Z)$ is a homotopy pullback of $(f_*,g_*) : B_{/K_2}(Y) \rightarrow B_{/K_2}(X) \times B_{/K_2}(X)$ and the diagonal $\Delta : B_{/K_2}(X) \rightarrow B_{/K_2}(X) \times B_{/K_2}(X)$.
\end{thm}
\begin{proof}
By Corollary \ref{cor Quillen B}, $C(Z)$ is a homotopy pullback of $(e_{-\infty}, e_{+\infty})_* : C(X^L) \rightarrow C(X^{K_2} \times X^{K_2})$ and $(f^{K_2}, g^{K_2})_* : C(Y^{K_2}) \rightarrow C(X^{K_2} \times X^{K_2})$. Thus the theorem follows from the following two commutative diagrams:{\tiny
$$\begin{CD}
C(X^L) @= C(X^L) @<{\simeq}<< C(X^{K_2}) @>{\simeq}>> B_{/K_2}(X)\\
@V{(e_{-\infty}, e_{+\infty})_*}VV @V{(e_{-\infty *}, e_{+\infty *})}VV @V{\Delta}VV @VV{\Delta}V\\
C(X^{K_2} \times X^{K_2}) @<{\simeq}<< C(X^{K_2}) \times C(X^{K_2}) @= C(X^{K_2}) \times C(X^{K_2}) @>{\simeq}>> B_{/K_2}(X) \times B_{/K_2}(X)
\end{CD}$$
}and
$$\begin{CD}
C(Y^{K_2}) @= C(Y^{K_2}) @>{\simeq}>> B_{/K_2}(Y)\\
@V{(f^{K_2}, g^{K_2})_*}VV @V{((f^{K_2})_* , (g^{K_2})_*)}VV @VV{(f_*,g_*)}V\\
C(X^{K_2} \times X^{K_2}) @>{\simeq}>> C(Y^{K_2}) \times C(Y^{K_2}) @>{\simeq}>> B_{/K_2}(X) \times B_{/K_2}(X).
\end{CD}$$
\end{proof}

\begin{eg} \label{eg 4.4}
Let $x: K_2 \rightarrow X$ be a bigraph homomorphism. Then $x$ is regarded as a point of $B_{/K_2}(X)$ and we consider $x$ as the basepoint of $B_{/K_2}(X)$. Note that $K_2^{K_2}$ is isomorphic to the graph ${\bf 1}$, the graph consisting of one looped vertex. Define the loop space construction $\Omega_{/K_2}(X,x)$ by the pullback diagram
$$\begin{CD}
\Omega_{/K_2}(X,x) @>>> X^L\\
@VVV @VVV\\
{\bf 1} @>{(x,x)}>> X^{K_2} \times X^{K_2}.
\end{CD}$$
Then Theorem \ref{thm 4.3} implies that $C(\Omega_{/K_2}(X,x)) \simeq \Omega(B_{/K_2}(X),x)$. Note that a vertex of $\Omega_{/K_2}(X,x)$ is a bigraph homomorphism $\gamma : L_{-\infty, + \infty} \rightarrow X$ such that $\gamma(2k) = x(0)$ and $\gamma(2k+1) = x(1)$ if $|k|$ is sufficiently large. Two vertices $\gamma$ and $\gamma'$ are adjacent if $(\gamma \times \gamma')(E(L_{-\infty, +\infty})) \subset E(X)$.

We recall the definition of 2-fundamental groups \cite{Matsushita 1}. Let $L_n = L_{0,n}$ and consider $L_n$ as a graph (not a bigraph). Let $(G,v)$ be a based graph. Here we do not assume that $v$ is a looped vertex. A graph homomorphism $\gamma : L_n \rightarrow G$ with $\gamma (0) = \gamma(n)$ is called a {\it loop with length $n$}. The length of a loop $\gamma$ is denoted by $l(\gamma)$. The set of loops of $(G,v)$ is denoted by $L(G,v)$. Consider the following conditions concerning a pair of loops $\gamma$ and $\gamma'$:

\begin{itemize}
\item[(1)] $l(\gamma') = l(\gamma) + 2$ and there is $x \in \{ 0,1,\cdots, l(\gamma)\}$ such that $\gamma (i) = \gamma'(i)$ for $i \leq x$ and $\gamma'(i + 2) = \gamma(i)$ for $i \geq x$. In particular, $\gamma'(x) = \gamma(x) = \gamma'(x+2)$.

\item[(2)] $\gamma$ and $\gamma'$ have the same length $n$, and $(\gamma \times \gamma')(E(L_n)) \subset E(G)$.
\end{itemize}
We write $\simeq$ the equivalence relation generated by (1) and (2). Let $\pi_1^2(G,v)$ be the set $L(G,v)/\simeq$ of equivalence classes of $\simeq$, and we call it the 2-fundamental group of $(G,v)$. The group structure of $\pi_1^2(G,v)$ is given by the concatenation of loops. 

By the definition of $\simeq$, we have the group homomorphism
$$\pi_1^2(G,v) \longrightarrow \Z_2, \; [\varphi] \longmapsto l(\varphi) \; {\rm modulo} \; 2.$$
The even part $\pi_1^2(G,v)_{ev}$ is the kernel of the above homomorphism. In other words, an element of $\pi_1^2(G,v)$ is an equivalence class $\alpha$ of $\simeq$ such that the parity of the length of a representative of $\alpha$ is even.

Let $(X,x)$ be a based bigraph. We want to show that $\pi_1^2(X,x(0)) \cong \pi_0(\Omega_{/K_2}(X), x)$. For a loop $\gamma : L{2n} \rightarrow X$ of $(X,x(0))$, define $\Phi(\gamma) \in \Omega_{/K_2}(X,x)$ as follows:
$$\Phi(\gamma) = \begin{cases}
\gamma(k) & (0 \leq k)\\
x(j) & ({\rm otherwise}, j = 0,1, k = j \; {\rm mod. 2})
\end{cases}$$
We want to show that if $\gamma \simeq_2 \gamma'$, then $\Phi(\gamma)$ and $\Phi(\gamma')$ belong to the same component of $\Omega_{/K_2}(X,x)$. If $\gamma$ and $\gamma'$ satisfy the condition (2) above, then $\Phi(\gamma)$ and $\Phi(\gamma')$ are adjacent in $\Omega_{/K_2}(X,x)$. Suppose that $\gamma$ and $\gamma'$ satisfy the condition (1). Let $\tilde{\gamma} : L_{2n+2} \rightarrow X$ be the extension of $\gamma$ which maps $2n + 2 - i$ to $x(i)$ for $i = 0, 1$. Then $\Phi(\gamma) = \Phi(\tilde{\gamma})$. Next let $\tilde{\gamma}'$ be the loop of $(X,x(0))$ defined by $\tilde{\gamma}' (i) = \gamma' (i)$ if $i \neq x+1$ and $\tilde{\gamma}' (x+1) = \gamma (x+1) = \gamma'(x+3)$. Then this $\gamma'$ and $\tilde{\gamma}'$ satisfy the condition (2) since they only differ at one point. Thus $\Phi(\tilde{\gamma}')$ and $\Phi(\gamma')$ are adjacent. It is easy to see that $\Phi(\tilde{\gamma}')$ and $\Phi(\gamma)$ belong to the same component of $\Omega_{/K_2}(X,x)$ by iterating the modification illustrated in Figure 1. Thus we have a correspondence from $\pi_1^2(X,x(0))_{ev}$ to $\pi_0(\Omega_{/K_2}(X,x))$. It is clear that $\Phi$ is bijective.

\begin{center}
\begin{picture}(270,60)(0,10)
\put(40,20){\line(1,2){20}} \put(60,60){\line(2,-1){40}} \put(0,40){\line(2,-1){36}} \put(36,22){\line(1,2){20}} \qbezier(56,62)(40,20)(40,20)

\put(33,33){\vector(1,2){8}} \put(33,33){\line(-2,1){16}}

\put(120,40){\vector(1,0){30}}

\put(170,40){\line(2,-1){40}} \put(210,20){\line(1,2){18}} \put(228,56){\line(2,-1){40}} \put(230,60){\line(2,-1){40}} \qbezier(268,36)(268,36)(230,60)
\end{picture}

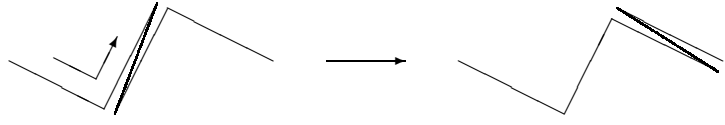
{\bf Figure 1.}
\end{center}

Suppose that $v$ is not an isolated vertex of a graph $G$ and let $w$ be a vertex adjacent to $v$. Recall that we want to show
$$\pi_1(N(G),v) \cong \pi_1^2(G,v)_{ev}$$
(see Section 2 for the definition of $N(G)$). Define $x : K_2 \rightarrow G$ by $x(0) = v$ and $x(1) = w$, and let $\tilde{x} = K_2 \times x$. Thus we have
\begin{eqnarray*}
\pi_1^2(G,v)_{ev} & \cong & \pi_1^2(K_2 \times G, (0,v)) \cong  \pi_0(\Omega_{/K_2} (K_2 \times G, \tilde{x})) \cong \pi_0(\Omega (B_{/K_2}(K_2 \times G), \tilde{x})) \\
& \cong & \pi_0(\Omega(B(G),x)) \cong \pi_1(B(G),x) \cong \pi_1(N(G),v).
\end{eqnarray*}
The verification of $\pi_1^2(G,v)_{ev} \cong \pi_1^2(K_2 \times G, (0,v))$ is straightforward, or found in \cite{Matsushita 1}. Thus Theorem \ref{thm Matsushita} follows.
\end{eg}

Let $X$ be a $\Z_2$-space. Recall that the free loop space of $X$ is denoted by $\L X$ and the space of $\Z_2$-maps from $S^1$ to $X$ is denoted by $\L' X$..

\begin{eg} \label{eg 4.5}
Let $X$ be a bigraph. Define the graph $\L X$ by the following pullback diagram:
$$\begin{CD}
\L X @>>> X^L\\
@VVV @VVV\\
X^{K_2} @>{\Delta_X^{K_2}}>> X^{K_2} \times X^{K_2}.
\end{CD}$$
Theorem \ref{thm 4.3} implies that
$$C(\L X) \simeq \L (B_{/K_2}(X)).$$
Define $r'_{n} :C_{n+2} \rightarrow C_n$ by $r'_n(i) = i$ for $i \leq n$ and $r'_n(n+1) = n-1$. We consider the colimit of $X^{C_{2n}}$ by $(r'_n)^*$. Then we have
$$C(\L X) \cong {\rm colim}_{n \rightarrow +\infty} C(X^{C_{2n}}) \simeq {\rm colim}_{n \rightarrow +\infty} \Hom_{/K_2}(C_{2n}, X).$$
For the last homotopy equivalence, see Lemma \ref{lem 3.2.2} and Proposition 15.10.12 of \cite{Hirschhorn} for example. Thus if $G$ is a graph, we have that
\begin{eqnarray*}
{\rm colim}_{n \rightarrow +\infty} \Hom(C_{2n}, G) &=& {\rm colim}_{n \rightarrow \infty} \Hom_{/K_2}(C_{2n} , K_2 \times G)\\
& \simeq & \L (B_{/K_2}(K_2 \times G)) = \L (B(G)).
\end{eqnarray*}
If we regard $X^{K_2}$ as a $\Z_2$-graph by the trivial $\Z_2$-action, then the diagonal $X^{K_2} \rightarrow X^{K_2} \times X^{K_2}$ is $\Z_2$-equivariant and hence we have
$${\rm colim}_{n \rightarrow +\infty} \Hom(C_{2n}, G) \simeq_{\Z_2} \L(B(G)).$$
\end{eg}

\begin{eg} \label{eg 4.6}
Let $\alpha_{K_2}$ be the involution of $K_2$ which flips the edge. Let $X$ be a bigraph with an odd involution $\alpha$. Then we have an involution $\alpha'$ of $X^{K_2}$ defined by $\alpha'(f) = \alpha_{K_2} \circ f \circ \alpha$. Regard $C(X^{K_2})$ as a $\Z_2$-space by this involution. On the other hand, we have an involution $\alpha''$ of $B_{/K_2}(X)$ defined by $\alpha'' (\sigma , \tau) = (\alpha (\tau), \alpha(\sigma))$. Then it is straightforward to see $B_{/K_2}(X) \simeq_{\Z_2} C(X^{K_2})$.

Define the graph $\L' X$ by the pullback diagram
$$\begin{CD}
\L' X @>>> X^L\\
@VVV @VVV\\
X^{K_2} @>{({\rm id}_X, \alpha')}>> X^{K_2} \times X^{K_2}.
\end{CD}$$
By the same way of the proof of Theorem \ref{thm 4.3}, we have that $C(\L' X)$ is the homotopy pullback of $({\rm id}, \alpha'') : B_{/K_2}(X) \rightarrow B_{/K_2}(X) \times B_{/K_2}(X)$ and the diagonal map $B_{/K_2}(X) \rightarrow B_{/K_2}(X) \times B_{/K_2}(X)$. It is easy to see that this is homotopy equivalent to $\L' B_{/K_2}(X)$. On the other hand, we have
$$C(\L' X) \simeq {\rm colim}_{n \rightarrow +\infty} C((X / \alpha )^{C_{2n+1}}) \simeq {\rm colim}_{n \rightarrow + \infty} \Hom(C_{2n+1}, X/\alpha).$$
Thus we have
$${\rm colim}_{n \rightarrow +\infty}\Hom(C_{2n + 1}, G) = \L' (B_{/K_2}(K_2 \times G)) \simeq \L'(B(G)).$$

If we regard $X^{K_2}$ as a $\Z_2$-graph by $\alpha^{K_2}$, then the map $({\rm id}_{X^{K_2}}, \alpha')$ is $\Z_2$-equivariant and hence we have
$${\rm colim}_{n \rightarrow +\infty} \Hom(C_{2n+1}, G) \simeq_{\Z_2} \L'(B(G)).$$
\end{eg}

Combining Example \ref{eg 4.5} and Example \ref{eg 4.6}, we have the proof of Theorem \ref{thm Schultz}.

\vspace{2mm} \noindent {\bf Acknowledgement.} The author thanks Shouta Tounai. He carefully read the draft of the paper and gave me helpful comments. The author is supported by the Grand-in-Aid for Scientific Research (KAKENHI 28-6304).

\end{document}